\documentclass[11pt,a4paper]{amsart}
\usepackage{amsfonts}
\usepackage{amsmath, amsthm}
\usepackage{amssymb}
\usepackage{latexsym}
\usepackage{exscale}

\usepackage[colorlinks=true, pdfstartview=FitH, linkcolor=blue, citecolor=red, urlcolor=blue]{hyperref} 
\allowdisplaybreaks

\newtheorem{thm}{Theorem}[section]
 \newtheorem{cor}[thm]{Corollary}
 
 \newtheorem{prop}[thm]{Proposition}
 \theoremstyle{definition}
 
 \theoremstyle{remark}
 \newtheorem{rem}[thm]{Remark}
 
 \numberwithin{equation}{section}

\def\be#1 {\begin{equation} \label{#1}}
\newcommand{\ee}{\end{equation}}

\newcommand{\R}{\mathbb{R}}

\newcommand{\N}{\mathbb{N}}
\newcommand{\la}{\lambda}

\def\XXint#1#2#3{{\setbox0=\hbox{$#1{#2#3}{\int}$}
     \vcenter{\hbox{$#2#3$}}\kern-.5\wd0}}

\date{\today}


\begin{document}
\title[Restriction estimates]{Restriction estimates via the derivatives of the  heat semigroup and connexion with dispersive estimates} 
\author{Fr\'ed\'eric Bernicot \& El Maati Ouhabaz}
\thanks{The two authors are partly supported by the ANR under the project "Harmonic Analysis at its Boundaries" no. ANR-12-BS01-0013. The first author is also partly supported by the ANR under the project AFoMEN no. 2011-JS01-001-01}
\maketitle

\begin{abstract}{We consider an abstract non-negative self-adjoint operator $H$ on an $L^2$-space. We derive a characterization  for the restriction estimate
$\| dE_H(\lambda) \|_{L^p \to L^{p'}} \le C \lambda^{\frac{d}{2}(\frac{1}{p} - \frac{1}{p'}) -1}$ 
 in terms of  higher order derivatives of the semigroup $e^{-tH}$. We provide an alternative proof  of a result in \cite{COSY} which asserts 
 that dispersive estimates imply restriction estimates. We also prove $L^p-L^{p'}$ estimates for the derivatives of the spectral resolution of $H$.}
\end{abstract}

\section
{Introduction and main results}\label{S1}

Let $(X,\mu)$ be a measured space. That is $X$ is a non-empty set endowed with a positive measure $\mu$. We  consider a non-negative self-adjoint operator $H$  on $L^2 = L^2(X, \mu)$. We denote by $dE_H$ the spectral resolution of $H$. Since we will be interested in $L^p-L^{p'}$  estimates for $dE_H(\la)$ we shall assume throughout this note that the spectrum of $H$ is continuous.  The $L^p-L^{p'}$ norm will be denoted by $ \|Â dE_H(\lambda) \|_{L^p \to L^{p'}}$ and 
$p'$ is  the conjugate number of $p$.  
%%%%%%%%%%%%%%%%%%

We first discuss the Euclidean Laplacian. Suppose that $X = \R^d$ and $H = -Â \Delta$ (the positive Laplace operator) on $L^2(\R^d)$. It is a well-known fact that as a consequence of the Stein-Tomas estimates for the restricted Fourier transform to the unit sphere, the spectral measure $dE_{-\Delta}(\la)$ is a bounded operator from $L^p$  into 
$L^{p'}$ for all $p \le \frac{2d + 2}{d + 3}$. In  addition,
$$\|Â dE_{-\Delta} (\lambda) \|_{L^p \to L^{p'}} \le C \lambda^{\frac{d}{2}(\frac{1}{p} - \frac{1}{p'}) -1}, \, \, \la > 0.$$
Such estimate is sometimes referred to as the $(p,2)$ restriction estimate of Stein-Tomas.  We refer to the introductions of the papers \cite{GHS} and 
\cite{COSY} for more details about this. 

The above  restriction estimate was extended to the setting of  asymptotically conic manifolds in \cite{GHS}. In the paper   \cite{COSY}  the restriction estimate  \begin{equation}\label{Rest}
\|Â dE_H(\lambda) \|_{L^p \to L^{p'}} \le C \lambda^{\frac{d}{2}(\frac{1}{p} - \frac{1}{p'}) -1},
\end{equation}
was studied in an abstract setting. (Here  $d$ is any positive constant).  It is also proved there that (\ref{Rest}) holds for several operators. 
 
\bigskip

 One of the  aims  of this note is to prove other characterizations of (\ref{Rest}) in an abstract setting.  
In the following result  we show  that the restriction estimate for $H$ can be  characterized in terms of higher order derivatives of the corresponding semigroup $e^{-tH}$. More precisely, 

\begin{thm}\label{th1}
Let $d$ be a positive constant and fix $p \in [1, 2)$.  The following assertions are equivalent.\\
1) The  restriction estimate (\ref{Rest}) holds for every $\lambda>0$;\\
2) There exists a positive constant $C$ such that 
\begin{equation}\label{Deriv}
\| H^N e^{-tH} \|_{L^p \to L^{p'}} \le C (N-1)! N^{\frac{d}{2}(\frac{1}{p} - \frac{1}{p'})} t^{-N - \frac{d}{2}(\frac{1}{p} - \frac{1}{p'})},
\end{equation} 
for all $t > 0$ and all $N \in \N$; \\
3) There exists a positive constant $C$ such that 
 \begin{equation}\label{ResF}
\| F(H) \|_{L^p \to L^{p'}} \le C R^{\frac{d}{2}(\frac{1}{p} - \frac{1}{p'})} \int_\R | F(s) | \frac{ds}{s},
\end{equation}
for all $R > 0$ and bounded measurable function $F$ with supported in $[0, R]$. 
\end{thm}

The main novelty here is the characterization of (\ref{Rest}) by (\ref{Deriv}). The equivalence of (\ref{Rest})  and (\ref{ResF}) is in the spirit of Proposition 2.4, Section 2.2 in \cite{COSY}. Note however that in contrast to that  proposition in \cite{COSY} we do not assume here that the volume of balls in $X$ 
is polynomial. Moreover, the $L^1$ norm in the RHS in (\ref{ResF}) is taken {\it w.r.t.}  $\frac{ds}{s}$ rather than $ds$ as in \cite{COSY} and it is obvious that 
$$\int_0^R | F(s) | \frac{ds}{s} = \int_0^1 | F(Rs) | \frac{ds}{s}  \ge  \int_0^1 | F(R s)| ds.$$

\bigskip

One of the main ingredients in the proof of Theorem \ref{th1}  is the following result which expresses the spectral  measure in terms of the semigroup.  We denote by 
$\langle ., . \rangle$ the scalar product of $L^2$. We have 

\begin{thm} \label{th2} 
Consider a bounded and uniformly $\rho$-H\"older function $\phi$ (for some $\rho\in(0,1]$). Then for every  $f,g\in L^2$ we have
$$ \lim_{N\rightarrow \infty} \frac{1}{(N-1)!} \int_0^\infty \phi(s^{-1}) \langle ((N-1)sH)^N e^{-s(N-1)H} f,g\rangle \frac{ds}{s}= \langle \phi(H)f,g\rangle.$$
\end{thm}

A useful consequence of the latter theorem is the following equality  for the derivatives  of $dE_H(\la)$  (in which the limit has to be understood in the weak sense)
$$ \frac{d^k}{d \la^k} dE_H(\la) = \lim_{N\to \infty} \  \  \frac{1}{N!}  \frac{d^k}{d \la^k} \left[ \la^{-1} (N \la^{-1}H)^{N+1} e^{-N \la^{-1}H} \right], $$
for  $\la > 0$. 

\medskip

As an application we show that dispersive estimates for $H$ imply the restriction estimate (\ref{Rest}) as well as $L^p-L^{p'}$ estimates for the derivatives 
$\frac{d^k}{d \la^k} dE_H(\la)$ for $k \le d/2 -1$. The estimates for the derivatives is new whereas the case $k = 0$ was already proved  in \cite{COSY}.

\medskip

We finish this introduction by explaining why it is interesting to prove the restriction estimate (\ref{Rest}).  Let us assume now that $(X,\mu)$ is equipped with a metric $\rho$ and assume that for every $x \in X$, $r > 0$, the volume $\mu(B(x,r))$ of the open ball $B(x,r)$ satisfies
$$c_1 r^d \le \mu(B(x,r)) \le c_2 r^d,$$
where $c_1$ and $c_2$ are positive constants. Suppose in addition that $H$ satisfies the finite speed of propagation property, that is the support of the kernel of
$\cos(t\sqrt{H})$ is contained in $\{ (x,y ) \in X \times X, \rho(x,y) \le t \}$.   Under these assumptions, it  is proved in \cite{COSY} (see also \cite{GHS} for the first assertion) that the restriction estimate implies sharp spectral multiplier theorems. More precisely, 

\begin{thm}\label{th3} Suppose that  the restriction estimate (\ref{Rest}) holds for some fixed $p \in [1, 2)$. Then the following assertions hold. 
\begin{itemize}
\item[(i)] {\bf Compactly supported multipliers:}  Let   $F$ be  an even function with support  in $ [-1,1]$
 and $F \in W^{\beta,2}(\R)$ for some $\beta>d(1/p-1/2)$. Then $F(H)$ is bounded on $L^p(X)$, and
$$ \sup_{t>0}\|F(tH)\|_{L^p \to L^p} \leq
C\|F\|_{W^{\beta,2}}.$$

\item[(ii)]{\bf  General  multipliers:} Suppose that    $F$ is an even  bounded  Borel
function which satisfies  $\sup_{t>0}\|\eta (\cdot) F(t \cdot)  \|_{W^{\beta, 2}}<\infty $ for some
$\beta>\max\{d(1/p-1/2),1/2\}$
 and some non-trivial  function $\eta \in C_c^\infty(0,\infty)$. Then
$F(H)$ is bounded on $L^r(X)$ for all $p<r<p'$.
In addition,
\begin{eqnarray*}
   \|F(H)  \|_{L^r\to L^r}\leq    C_\beta\Big(\sup_{t>0}\|\eta (\cdot) F(t \cdot) \|_{W^{\beta, 2}}
   + |F(0)|\Big).
\end{eqnarray*}
\end{itemize}
\end{thm}
A version of this  theorem for general doubling spaces is proved  in \cite{COSY}. One can apply 
  Theorem \ref{th3}  to prove summability results for Bochner-Riesz means on $L^p$-spaces.

%%%%%%%%%%%%%%%%%%%%%%%%%%%%%

\section{Proofs}\label{S2} 
We start with the proof of Theorem \ref{th2}. We shall write $\lesssim$ for $\le$ up to a non relevant constant $C$. 
\begin{proof}[Proof of Theorem \ref{th2}]  
Let us set $ c_N^{-1}:= \int_0^\infty x^N e^{-x} \frac{dx}{x} =(N-1) !$. 
By polarization, it suffices  to prove that for every function $f\in L^2$
\begin{equation}\label{eq2.1}
 \lim_{N\rightarrow \infty} c_N\int_0^\infty \phi(s^{-1}) \langle ((N-1)sH)^N e^{-s(N-1)H} f,f\rangle \frac{ds}{s} = \langle \phi(H)f,f\rangle.
 \end{equation}
We have 
\begin{align} 
& c_N\int_0^\infty \phi(s^{-1}) \langle ((N-1)sH)^N e^{-s(N-1)H} f,f\rangle \frac{ds}{s}\\
  & \hspace{2cm} = c_N \int_0^\infty\int_0^\infty \phi(s^{-1}) ((N-1) s\lambda)^N e^{-(N-1)s \lambda} \langle dE_\lambda f,f\rangle \frac{ds}{s} \nonumber \\
& \hspace{2cm} = \int_0^\infty \mu_N(\lambda) \langle dE_\lambda f,f\rangle, \label{eq:mun}
\end{align}
where 
$$ \mu_N(\lambda) := c_N \int_0^\infty  \phi(s^{-1}) ((N-1)s\lambda)^N e^{-(N-1)s\lambda} \frac{ds}{s}.$$
Due to the constant $c_N$, it is clear  that the continuous function $\mu_N$ is  bounded by $\|\phi\|_{L^\infty}$.
Therefore, it is enough to prove that $\mu_N(\lambda)$  converges to the function $\phi(\lambda)$ for all $\lambda > 0$ and then conclude by the dominated convergence theorem. \\
Taking the difference yields
\begin{align*}
\left|\mu_N(\lambda)-\phi(\lambda)\right| & = c_N \left|\int_0^\infty  \left[\phi(s^{-1})-\phi(\lambda)\right] ((N-1)s\lambda)^N e^{-(N-1)s\lambda} \frac{ds}{s}\right| \\
& \leq  \int_0^\infty  \left|\phi(s^{-1})-\phi(\lambda)\right| c_N ((N-1)s\lambda)^N e^{-(N-1)s\lambda} \frac{ds}{s},
\end{align*}
Using Stirling's formula
$$ c_N=\left[(N-1) !\right]^{-1} \simeq \left(\frac{e}{N-1}\right)^{N-1} (2\pi N)^{-\frac{1}{2}}$$
 we obtain for large enough $N$ and uniformly with respect to $\lambda$ 
\begin{align*}
\left|\mu_N(\lambda)-\phi(\lambda)\right| & \lesssim  N^{\frac{1}{2}}\int_0^\infty  \left|\phi(s^{-1})-\phi(\lambda)\right|  (s\lambda)^N e^{-(N-1)(s\lambda-1)} \frac{ds}{s} \\
 & \lesssim  N^{\frac{1}{2}}\int_0^\infty  \left|\phi(s^{-1})-\phi(\lambda)\right|  (s\lambda e^{-(s\lambda-1)})^{N-1} \lambda ds.
\end{align*}
Next we  decompose the integral for $s\lambda\leq  u_N$, $u_N<s\lambda <v_N$ and $s\lambda \geq v_N$, obtaining three terms $I$, $II$ and $III$ and where $u_N<1$ and $v_N>1$ will be suitably chosen later (around $1$).
For the first term, we have (since $x\rightarrow xe^{-(x-1)}$ is non-decreasing for $x\in(0,1)$)
\begin{align*}
 I & \lesssim  N^{\frac{1}{2}} (u_Ne^{-(u_N-1)})^{N-1} \lambda \|\phi\|_{L^\infty} \left(\int_{s\lambda\leq u_N}   ds\right)  \\
  & \lesssim N^{\frac{1}{2}} (u_Ne^{-(u_N-1)})^{N-1}  \|\phi\|_{L^\infty}.
\end{align*}
For the third term; we similarly have (since $x\rightarrow xe^{-(x-1)}$ is decreasing for $x\in(1,\infty)$)
\begin{align*}
 III & \lesssim  N^{\frac{1}{2}} (v_Ne^{-(v_N-1)})^{N-2} \left(\int_{s\lambda\geq v_N} (s\lambda) e^{-(s\lambda-1)} \lambda ds \right)  \|\phi\|_{L^\infty} \\
  & \lesssim N^{\frac{1}{2}} (v_Ne^{-(v_N-1)})^{N-2}  \|\phi\|_{L^\infty}.
\end{align*}
About the second term, using $\omega$ the uniform modulus of continuity of $\phi$ (and that $xe^{-(x-1)}\leq 1$ for every $x>0$),  it follows that 
\begin{align*}
 II & \lesssim  N^{\frac{1}{2}} \int_{u_N\leq s \lambda \leq v_N} \omega(\frac{v_N-u_N}{s}) \lambda ds \\
  & \lesssim N^{\frac{1}{2}} \omega(2\lambda(v_N-u_N))|v_N-u_N|,
\end{align*}
where we used that $u_N,v_N$ are around the value $1$.
Finally, we deduce that for every $\lambda$
\begin{align*}
 \left|\mu_N(\lambda)-\phi(\lambda)\right| & \lesssim N^{\frac{1}{2}} (u_Ne^{-(u_N-1)})^{N-1} \\
 & +N^{\frac{1}{2}} (v_Ne^{-(v_N-1)})^{N-2}+N^{\frac{1}{2}} \omega(\lambda(v_N-u_N))|v_N-u_N|.
\end{align*}

Now, let us write 
$$ u_N:= 1- \epsilon_N \qquad \textrm{and} \qquad v_N:= 1+ \epsilon_N,$$
with $\epsilon_N \to 0$ as $N \to \infty$.
We note that by a second order expansion, 
\begin{align*}
 \log\left(N^{\frac{1}{2}}\left(u_Ne^{-(u_N-1)}\right)^{N-1}\right) & = \frac{1}{2}\log(N) + (N-1)\left[\log(u_N)-(u_N-1)\right]  \\
 & = \frac{1}{2}\log(N) - (N-1)\left[\frac{1}{2}(u_N-1)^2 + {\mathcal O}(u_N-1)^3\right] \\
 & = \frac{1}{2}\log(N) - (N-1)\left[\frac{1}{2} \epsilon_N^2 + {\mathcal O}(\epsilon_N^3)\right] \\
 &   \xrightarrow[N\to \infty]{} -\infty
 \end{align*}
 provided 
\begin{equation} \lim_{N\to \infty} \frac{N}{\log(N)}\epsilon_N^2 = \infty. \label{eq:lim1} \end{equation}
In this case,  
$$ \lim_{N\to \infty} N^{\frac{1}{2}} (u_Ne^{-(u_N-1)})^{N-1} = 0.$$
Similarly , we have
$$ \lim_{N\to \infty} N^{\frac{1}{2}} (v_Ne^{-(v_N-1)})^{N-2} = 0,$$
and moreover since $\omega$ tends to $0$ at $0$ (with an order $\rho\in(0,1]$: $\omega(s)\lesssim s^\rho$), we can choose $\epsilon_N$ such that 
\begin{equation} \lim_{N\to \infty} N^{\frac{1}{2}} |v_N-u_N|\omega(2\lambda(v_N-u_N)) \lesssim \lim_{N\to \infty} N^{\frac{1}{2}}\epsilon_N \omega(4\lambda\epsilon_N) = 0. \label{eq:lim2}
 \end{equation}
Indeed, take $\epsilon_N$ such that 
$$ N\epsilon_N^2 = \log(N)^2,$$
(which is possible for large enough integer $N$) then $\epsilon_N$ tends to $0$ and (\ref{eq:lim1}) is satisfied. Moreover (\ref{eq:lim2}) follows from $\omega(s)\lesssim s^\rho$, due to the $\rho$-H\"older regularity of $\phi$.
For such $\epsilon_N$, we finally conclude to 
$$ \lim_{N\to \infty}  \left|\mu_N(\lambda)-\phi(\lambda)\right| = 0.$$
Then, using dominated convergence Theorem and then spectral theory in (\ref{eq:mun}) implies (\ref{eq2.1}). 
\end{proof}

\begin{cor} \label{cor:s}  
Let $\delta\in(0,1]$.  For every smooth function $\phi$ and every $L^2$-functions $f,g$, we have
\begin{align*}
&\int_0^\infty \phi(\lambda) \langle dE_H(\lambda) f,g\rangle d\lambda\\
 & = \lim_{N\to \infty} \frac{1}{(N-1)!}  \int_0^\infty \phi(\lambda)\langle ((N-\delta)\lambda^{-1}H)^N e^{-\lambda^{-1}(N-\delta)H}f,g\rangle \frac{d\lambda}{\lambda} \\
 & = \lim_{N\to \infty} \frac{1}{\Gamma(N+\delta)} \int_0^\infty \phi(\lambda)\langle (N\lambda^{-1}H)^{N+\delta} e^{-\lambda^{-1}NH}f,g\rangle \frac{d\lambda}{\lambda}. 
\end{align*}
\end{cor}

\begin{proof} The case $\delta=1$ is exactly the statement of Theorem \ref{th2}. For  $\delta\in(0,1)$, we follow  the same proof  and replace $(N-1)!=\Gamma(N)$ by  $\Gamma(N-\delta+1)$ (we recall that Stirling's formula remains valid for  the $\Gamma$ function, see (\ref{Stir}) below).
 \end{proof}

Making an integration by parts in Theorem \ref{th2}, we obtain a formula  for the derivatives $\frac{d^k}{d \la^k} dE_H(\la)$ in terms of the semigroup.  That is
\begin{cor}\label{cor2}
The following equality holds in the weak sense: for an integer $k\geq 1$ 
$$  \frac{d^k}{d \la^k} dE_H(\la) = \lim_N \frac{1}{N!}  \frac{d^k}{d \la^k} \left[ \la^{-1} (N \la^{-1}H)^{N+1} e^{-N \la^{-1}H} \right]$$
for  $\la > 0$.
\end{cor}

\begin{proof}[Proof of Theorem \ref{th1}] 
We first prove that $1)$ implies $2)$.  Suppose that (\ref{Rest}) is satisfied. For fixed $N$ we have 
\begin{align*}
\|Â H^N e^{-tH}  \|_{L^p \to L^{p'}} &= \left\|  \int_0^\infty \lambda^N e^{-t\lambda} dE_H(\lambda)  \right\|_{L^p \to L^{p'}}\\
&\lesssim \int_0^\infty \lambda^N e^{-t\lambda} \la^{\frac{d}{2}(\frac{1}{p} - \frac{1}{p'}) - 1} d\la\\
&= \left(\int_0^\infty u^{N + \frac{d}{2}(\frac{1}{p} - \frac{1}{p'}) -1} e^{-u}  du\right) t^{ -N - \frac{d}{2}(\frac{1}{p} - \frac{1}{p'})}\\
&= t^{ -N - \frac{d}{2}(\frac{1}{p} - \frac{1}{p'})} \Gamma(N + \frac{d}{2}(\frac{1}{p} - \frac{1}{p'})).
\end{align*}
Stirling's formula for the Gamma function (see \cite[Appendix A.6]{Gra})
\begin{equation} \label{Stir}
\Gamma(x) \simeq x^{x - \frac{1}{2}} e^{-x} \sqrt{2 \pi} \, \, \,   {\rm for } \, x>0,
\end{equation}
shows that 
$$\|Â H^N e^{-tH}  \|_{L^p \to L^{p'}}  \lesssim (N-1)! N^{\frac{d}{2}(\frac{1}{p} - \frac{1}{p'})} t^{-N - \frac{d}{2}(\frac{1}{p} - \frac{1}{p'})}.$$
This proves assertion $2)$. 

We now prove that $2)$ implies $3)$.  Suppose first that $F$ is  a H\"older continuous function with support in $[0, R]$. We apply Theorem \ref{th2} and obtain
\begin{align*}
\| F(H) \|_{L^p \to L^{p'}} &\le \limsup_N \frac{1}{(N-1)!} \int_{1/R}^\infty  F(s^{-1}) \| ((N-1)sH)^Ne^{-s (N-1)H} \|_{p-p'}  \frac{ds}{s}\\
& \lesssim  \limsup_N \int_{1/R}^\infty F(s^{-1}) ((N-1)s)^N   N^{\frac{d}{2}(\frac{1}{p} - \frac{1}{p'})} 
(s (N-1))^{-N - \frac{d}{2}(\frac{1}{p} - \frac{1}{p'})}  \frac{ds}{s}\\
& \lesssim \int_{1/R}^\infty F(s^{-1}) s^{- \frac{d}{2}(\frac{1}{p} - \frac{1}{p'})}  \frac{ds}{s}\\
&\le R^{\frac{d}{2}(\frac{1}{p} - \frac{1}{p'})} \int_{0}^R  F(s) \frac{ds}{s}.
\end{align*}
We have proved that 
\begin{equation}\label{eq2.2}
\| F(H) \|_{L^p \to L^{p'}}  \lesssim R^{\frac{d}{2}(\frac{1}{p} - \frac{1}{p'})} \int_{0}^R  F(s) \frac{ds}{s}.
\end{equation}
Next we  extend the latter  estimate  to all bounded functions $F$ with support in $[0,R]$.
 This can be achieved  by classical  approximation arguments. First assume that the support of $F$ is contained 
in $[\eta, R]$ for some $\eta > 0$ and apply (\ref{eq2.2}) to the convolution  $F_\epsilon = \rho_\epsilon * F$ 
by a mollifier  $\rho_\epsilon$. We  obtain 
$$\| F_\epsilon(H) \|_{L^p \to L^{p'}}  \lesssim (R+ \epsilon)^{ \frac{d}{2}(\frac{1}{p} - \frac{1}{p'})} \int_{\eta-\epsilon}^R  F_\epsilon(s) \frac{ds}{s}.$$
Since $| F_\epsilon(s)| \le \| F \|_\infty$  and  the support of $F_\epsilon$ is contained in $[\eta/2, R+ \eta/2]$ for $\epsilon < \eta/2$
one can apply  the dominated convergence theorem to the RHS of the previous inequality. We obtain (\ref{eq2.2}). Now for every bounded $F$ with support in $[0, R]$ we can apply (\ref{eq2.2}) to $\chi_{[\epsilon, R]} F$ and then let $\epsilon \to 0$. Assertion $3)$ is then proved. 

Finally we prove that $3)$ implies $1)$. In order to do this, we fix $\lambda > 0$ and $\epsilon \in (0, \la)$, and  apply $3)$ to 
$F(s) = \chi_{(\la- \epsilon, \la + \epsilon]}(s)$. It follows that 
\begin{align*}
\| \chi_{(\la- \epsilon, \la + \epsilon]}(H)  \|_{L^p \to L^{p'}} &\lesssim (\la + \epsilon)^{\frac{d}{2}(\frac{1}{p} - \frac{1}{p'})} \int_{\la - \epsilon}^{\la + \epsilon} \frac{ds}{s}\\
&= (\la + \epsilon)^{\frac{d}{2}(\frac{1}{p} - \frac{1}{p'})} [ \ln(\la + \epsilon) - \ln(\la - \epsilon)]\\
&\simeq (\la + \epsilon)^{\frac{d}{2}(\frac{1}{p} - \frac{1}{p'})} \frac{2 \epsilon}{\la}.
\end{align*} 
Hence 
$$\| \epsilon^{-1} \chi_{(\la- \epsilon, \la + \epsilon]}(H)  \|_{L^p \to L^{p'}}  \lesssim  (\la + \epsilon)^{\frac{d}{2}(\frac{1}{p} - \frac{1}{p'})} \la^{-1}.$$
We let $\epsilon \to 0$ and obtain assertion $1)$. 

\end{proof} 

\begin{rem} \ \\
\begin{itemize} 
 \item  In the proof of $1) \Rightarrow 2)$ we can take $N = 0$ and obtain that $1)$ implies 
$$ \| e^{-tH} \|_{L^p \to L^{p'}} \lesssim t^{-\frac{d}{2}(\frac{1}{p} - \frac{1}{p'})}, \, t > 0.$$
\item  Suppose that (\ref{Rest}) holds. Let $\alpha > 0$ and apply assertion $3)$ with $F(s^\alpha)$ to obtain
$$ \|Â F(H^\alpha) \|_{L^p \to L^{p'}} \lesssim R^{\frac{d}{2\alpha} (\frac{1}{p} - \frac{1}{p'})} \int_\R |F(s^\alpha)| \frac{ds}{s} = 
\alpha R^{\frac{ d}{2\alpha} (\frac{1}{p} - \frac{1}{p'})} \int_\R |F(s)| \frac{ds}{s}.$$ 
We conclude by Theorem \ref{th1} that 
$$\|Â dE_{H^\alpha}(\lambda) \|_{L^p \to L^{p'}} \le C  \la^{\frac{ d}{2\alpha } (\frac{1}{p} - \frac{1}{p'}) -1}, \, \la > 0.$$
In particular, for $\alpha=\frac{1}{2}$
$$ \|Â dE_{\sqrt{H}}(\lambda) \|_{L^p \to L^{p'}} \le C \la^{d (\frac{1}{p} - \frac{1}{p'}) -1}.$$
\item Assume that the heat semigroup $(e^{-tH})_{t>0}$ satisfies the classical  $L^p-L^{2}$ estimates $\|e^{-tH}\|_{L^p \to L^{2}} \lesssim t^{-\frac{d}{2}\left(\frac{1}{p}-\frac{1}{2}\right)}$ for every $t>0$ and some $p\in[1,2]$. Then we observe that for every integer $N\geq 3$
\begin{align*}
 \| H^N e^{-tH} \|_{L^p\to L^{p'}} & \leq   \| e^{-\frac{t}{N} H}\|_{L^2 \to L^{p'}}  \| H^N e^{-t(1-\frac{2}{N})H}\|_{L^2 \to L^2} \| e^{-\frac{t}{N} H}\|_{L^p \to L^2}  \\
 & \lesssim \left(\frac{t}{N}\right)^{-\frac{d}{2}\left(\frac{1}{p}-\frac{1}{p'}\right)} \left(\frac{N}{t(1-\frac{2}{N})}\right)^N e^{-N} \\
& \lesssim  t^{-N-\frac{d}{2}\left(\frac{1}{p}-\frac{1}{p'}\right)} N^{\frac{d}{2}\left(\frac{1}{p}-\frac{1}{p'}\right)} (Ne^{-1})^{N} \\
& \lesssim t^{-N-\frac{d}{2}\left(\frac{1}{p}-\frac{1}{p'}\right)} N^{\frac{d}{2}\left(\frac{1}{p}-\frac{1}{p'}\right)} (N-1)! \sqrt{N},
\end{align*}
where we used Stirling's formula to obtain the last inequality.  Therefore we see  that the gap between this very general estimate with the one required in Theorem \ref{th1} is an extra term of order $N^{\frac{1}{2}}$.
\end{itemize}
\end{rem}

\section{Restriction from dispersion }

In this section we show that dispersive estimates for the semigroup generated by $H$ imply restriction estimates and also $L^p-L^{p'}$ estimates for the derivatives 
$\frac{d^k}{d \la^k} dE_H(\la)$ up to some order.  The result  for the case $k = 0$ was already derived in \cite{COSY} by a different proof. The result for $k \ge 1$ seems to be new. 
 
\begin{prop} \label{prop2} Fix $1\leq p<\frac{2d}{d+2}=2_*$. Suppose  that the semigroup $(e^{-zH})_{z\in{\mathbb C}^+}$ satisfies  the following dispersive  estimates: 
\begin{equation}\label{disp-p}
 \|e^{-zH}\|_{L^p \to L^{p'}} \lesssim |z|^{-\frac{d}{2}\left(\frac{1}{p}-\frac{1}{p'} \right)}, 
 \end{equation}
uniformly in  $z$ such that  $\Re(z) > 0$. Then, for $\gamma>0$ 
$$ \|(NsH)^{N+\gamma} e^{-NsH}\|_{L^{p} \to L^{p'}} \lesssim s^{-\frac{d}{2}\left(\frac{1}{p}-\frac{1}{p'}\right)} N^\gamma (N-1)!,$$
 uniformly in  $N\geq 1$ and $s>0$ (the implicit constant only depends on $\gamma>0$). 
\end{prop}

\begin{rem} \label{rem} It follows from Cauchy's formula that (\ref{disp-p})  extends  to the derivatives of the semigroup as follows: for $k\geq 1$ and every $z\in{\mathbb C}^+$
$$ \|(\Re(z)H)^k e^{-zH}\|_{L^p \to L^{p'}} \lesssim |z|^{-\frac{n}{2}\left(\frac{1}{p}-\frac{1}{p'} \right)}.$$
\end{rem}

\begin{proof} By the functional calculus, we have
$$(NsH)^{N+\gamma} e^{-NsH} = N! N^N (sNH)^\gamma \int_{\Gamma_N} \frac{e^{-\zeta s H}}{(\zeta-N)^{N+1}} d\zeta,$$
where $\Gamma_N$ is the circle of center $N$ and of radius $N-1$. So we deduce that
\begin{align}
\|(NHs)^{N+\gamma} e^{-NsH}\|_{L^p \to L^{p'}} & \leq N^{N+\gamma} N!  \int_{\Gamma_N} \frac{\|(sH)^\gamma e^{-\zeta sH}\|_{L^p \to L^{p'}}}{|\zeta-N|^{N+1}} d\zeta \nonumber \\
& \leq N^{N+\gamma} N!  (N-1)^{-(N+1)} \int_{\Gamma_N} \|(sH)^\gamma e^{-\zeta sH}\|_{L^p \to L^{p'}} d\zeta \nonumber \\
& \lesssim N!  N^{\gamma} \int_{0}^{2\pi} \|(sH)^\gamma e^{-\zeta_\theta sH}\|_{L^p \to L^{p'}} d\theta, \label{eq:cauchy}
\end{align}
with $\zeta_\theta:=N+(N-1)e^{i\theta}$.
Writing (up to some numerical constant) with an integer $k\geq 1+\gamma$
$$ (sH)^\gamma e^{-\zeta sH} = s^\gamma \int_0^\infty   (tH)^k e^{-(t+s\zeta)H} \frac{dt}{t^{1+\gamma}}, $$
we deduce that (since $\gamma \in(0,k)$)
\begin{align*}
 \|(sH)^\gamma e^{-\zeta sH}\|_{L^p \to L^{p'}} & \lesssim s^\gamma \int_0^\infty  \frac{t^k}{(t+s\Re(\zeta))^k(|t+s\zeta|)^{\frac{d}{2}\left(\frac{1}{p}-\frac{1}{p'} \right)}} \frac{dt}{t^{1+\gamma}} \\
 & \lesssim s^\gamma (s|\zeta|)^{-\frac{d}{2}\left(\frac{1}{p}-\frac{1}{p'} \right)} \int_0^{\infty}  \left(\frac{t}{t+s\Re(\zeta)}\right)^k  \frac{dt}{t^{1+\gamma}} \\
  & \lesssim (s|\zeta|)^{-\frac{d}{2}\left(\frac{1}{p}-\frac{1}{p'} \right)} s^\gamma \left[ \int_0^{s\Re(\zeta)}  \left(\frac{t}{s\Re(\zeta)}\right)^k  \frac{dt}{t^{1+\gamma}} + \int_{s\Re(\zeta)}^\infty \frac{dt}{t^{1+\gamma}}\right] \\
   & \lesssim (s|\zeta|)^{-\frac{d}{2}\left(\frac{1}{p}-\frac{1}{p'} \right)} \Re(\zeta)^{-\gamma} \\
   & \lesssim (s|\zeta|)^{-\frac{d}{2}\left(\frac{1}{p}-\frac{1}{p'} \right)} 
 \end{align*}
 where we used that $\Re(\zeta)\geq 1$ and $|t+s\zeta| \geq |s\zeta|$.
Putting this estimate together with (\ref{eq:cauchy}) yields with $\sigma:=\frac{d}{2}\left(\frac{1}{p}-\frac{1}{p'} \right)$
\begin{align*}
& \|(NsH)^{N+\gamma} e^{-NsH}\|_{L^p \to L^{p'}}\\
 & \lesssim s^{-\sigma} N!  N^{\gamma} \int_{0}^{2\pi} \left|N+(N-1)e^{i\theta}\right|^{-\sigma} d\theta \\
& \lesssim  s^{-\sigma} N!  N^{\gamma} \int_{-1}^{1} \left( (N+(N-1)u)^2 +(N-1)^2(1-u^2)\right)^{-\sigma/2} \frac{du}{\sqrt{1-u^2}} \\
& \lesssim  s^{-\sigma} N!  N^{\gamma} \left[\int_{-1}^{0} \left(N^2(1+u)^2+(N-1)^2(1+u)\right)^{-\sigma/2} \frac{du}{\sqrt{1+u}} + \right. \\
 & \hspace{3cm}  \left. \int_{0}^{1} N^{-\sigma} \frac{du}{\sqrt{1-u}} \right] \\
& \lesssim  s^{-\sigma} N!  N^{\gamma} \left[\int_{0}^{(N-1)^2} (1+v)^{-\sigma/2} \frac{dv}{N\sqrt{v}} +N^{-\sigma}  \right] \\
& \lesssim  s^{-\sigma} N!  N^{\gamma-1}\simeq s^{-\sigma} (N-1) ! N^\gamma,
\end{align*}
where we used that $\sigma>1$ since $p<2_*$. 
\end{proof}

\begin{cor}  \label{cor1bis} 
Assume that the semigroup $(e^{-zH})_{z\in{\mathbb C}^+}$ satisfies the dispersive estimate (\ref{disp-p}) for some $1 \le p<\frac{2d}{d+2}=2_*$. Then we have 
$$ \| dE_{H}(\lambda) \|_{L^p \to L^{p'}} \lesssim \lambda^{\frac{d}{2}\left(\frac{1}{p}-\frac{1}{p'}\right)-1},  \,  \, \la > 0.$$
In addition for an integer $k \le d/2 -1$, if $1\leq p<\frac{2d}{d+2(k+1)}$ then 
$$ \left\| \frac{d^k}{d\la^k}  dE_{H}(\lambda) \right\|_{L^p \to L^{p'}} \lesssim \lambda^{\frac{d}{2}\left(\frac{1}{p}-\frac{1}{p'}\right)-(k+1)}, \, \,  \la > 0. $$
\end{cor}

\begin{proof} The first assertion follows immediately from Proposition \ref{prop2} and Theorem \ref{th1}. 
For the second assertion we give for simplicity a proof for $k = 1$, the general case follows by iteration.  By Corollary \ref{cor2}   we have in the weak sense 
\begin{align*}
 \frac{d}{d\la} dE_H(\lambda)  
 & = \lim_{N\to \infty} \frac{1}{N !}  \frac{d}{d\lambda} \left[ (N\lambda^{-1}H)^{N+1} e^{-\lambda^{-1}NH} \lambda^{-1} \right] \\
 & = \lim_{N\to \infty} \frac{1}{N !}  (NH)^{-1} \frac{d}{d\lambda} \left[ (N\lambda^{-1}H)^{N+2} e^{-\lambda^{-1}NH} \right].
\end{align*}
Following Proposition \ref{prop2}, we write 
$$ (N\lambda^{-1} H)^{N+2} e^{-N\lambda^{-1} H} = (N+2)! N^{N+2} \int_{\Gamma_N} \frac{e^{-\zeta \lambda^{-1} H}}{(\zeta-N)^{N+3}} d\zeta$$
and so
$$ \frac{d}{d\lambda} \left[ (N\lambda^{-1} H)^{N+2} e^{-N\lambda^{-1} H} \right] = (N+2)! N^{N+2} \lambda^{-2} \int_{\Gamma_N} \frac{\zeta H e^{-\zeta \lambda^{-1} H}}{(\zeta-N)^{N+1}} d\zeta.$$
Hence, 
$$ \frac{1}{N !}  (NH)^{-1} \frac{d}{d\lambda} \left[ (N\lambda^{-1}H)^{N+2} e^{-\lambda^{-1}NH} \right] = \frac{(N+2)!}{(N-1)!}  N^{N} \lambda^{-2} \int_{\Gamma_N} \frac{\zeta e^{-\zeta \lambda^{-1} H}}{(\zeta-N)^{N+1}} d\zeta.$$
Using the dispersive estimate, it follows that 
\begin{align*}
 &\frac{1}{N!}  (NH)^{-1} \left\| \frac{d}{d\lambda} \left[ (N\lambda^{-1}H)^{N+2} e^{-\lambda^{-1}NH} \right] \right\|_{L^p \to L^{p'}}\\
  & \lesssim  \frac{(N+2)!}{(N-1)!}  N^{N} \lambda^{-2} \int_{\Gamma_N} \frac{|\zeta| \|e^{-\zeta \lambda^{-1} H}\|_{L^p \to L^{p'}} }{|\zeta-N|^{N+1}} d\zeta \\
& \lesssim  \frac{(N+2)!}{(N-1)!}  \frac{N^{N}}{(N-1)^{N+1}} \lambda^{-2} \int_{\Gamma_N} |\zeta|^{1-\sigma} \lambda^{\sigma} d\zeta ,
\end{align*}
with $\sigma:=\frac{d}{2}\left(\frac{1}{p}-\frac{1}{p'} \right)$. Such integral was already computed in Proposition \ref{prop2} and is uniformly bounded as soon as $1-\sigma<-1$.
This gives the desired estimate for  $\frac{d}{d\la} E_{H}(\lambda)$.
\end{proof}

\vspace{.5cm}

\noindent
\emph{Fr\'ed\'eric Bernicot,} CNRS - Universit\'e de Nantes, Laboratoire Jean Leray. 2, rue de la Houssini\`ere, 44322 Nantes cedex 3.  France,\\ 
\texttt{Frederic.Bernicot@univ-nantes.fr}

\quad\\
\noindent
\emph{El Maati Ouhabaz,} Institut de Math\'ematiques (IMB), Univ.\  Bordeaux, 351, cours de la Lib\'eration, 33405 Talence cedex, France,\\ 
\texttt{Elmaati.Ouhabaz@math.u-bordeaux1.fr}


\vspace{.5cm}

\begin{thebibliography}{} 
 
 \bibitem{COSY} P. Chen, E.M. Ouhabaz, A. Sikora and L. Yan,
 \newblock Restriction estimates, sharp spectral multipliers and endpoint estimates for Bochner-Riesz means. 
 \newblock  Submitted (2011), http://arxiv.org/abs/1202.4052.

\bibitem{Gra}
 L. Grafakos, 
\newblock {\it Classical Fourier Analysis, Second Edition.}  
\newblock Graduate Texts in Mathematics {\bf 249}, Springer, New York, NY, 2009.

 \bibitem{GHS}
 C. Guillarmou, A. Hassel and A. Sikora, 
\newblock Restriction and spectral multiplier theorems on asymptotically conic manifolds,
\newblock to appear in {\it Analysis and PDE}, http://arxiv.org/abs/1012.3780.

\bibitem{Ho1}
L. H{\"o}rmander,
{\it The analysis of linear partial differential operators}, {\bf I, II}.
 Springer-Verlag, Berlin, 1983.

% \bibitem{St2} E.M. Stein,  {\it Harmonic analysis: Real variable
%methods, orthogonality and oscillatory integrals}. With the assistance of Timothy S. Murphy. Princeton Mathematical Series, {\bf 43}.
%Monographs in Harmonic Analysis, III. Princeton University Press, Princeton, NJ, 1993
%Princeton Univ. Press, Princeton, NJ, 1993.


\end{thebibliography}
\end{document}